\documentclass[a4paper,10pt]{article}
\title{Generalized Silver and Miller measurability}
\author{Giorgio Laguzzi}

\usepackage{fontenc}
\usepackage{amsmath}
\usepackage{amsfonts}
\usepackage{amssymb}
\usepackage{amsthm}
\usepackage{graphicx}
\usepackage[english]{babel}

\newtheorem{definition2}{Definition}
\newtheorem{fact}[definition2]{Fact}
\newtheorem{lemma}[definition2]{Lemma}
\newtheorem{claim}[definition2]{Claim}
\newtheorem{remark2}[definition2]{Remark}

\newtheorem*{Mobservation2}{Main Observation}
\newtheorem{example2}[definition2]{Example}
\newtheorem{corollary}[definition2]{Corollary}

\newenvironment{definition}{\begin{definition2} \upshape}{\end{definition2}}
\newenvironment{remark}{\begin{remark2} \upshape}{\end{remark2}}

\newcommand{\cohen}{\poset{C}}
\newcommand{\conc}{\smallfrown}
\newcommand{\club}{\textsc{Cub}}
\newcommand{\clubmiller}{\miller^\mathsf{club}}
\newcommand{\clubsilver}{\silver^\mathsf{club}}
\newcommand{\clubsacks}{\sacks^\mathsf{club}}
\newcommand{\DDelta}{\mathbf{\Delta}}

\newcommand{\dom}{\text{dom}}

\newcommand{\enfa}{\textit}

\newcommand{\ideal}{\mathcal}
\newcommand{\ifif}{\Leftrightarrow}
\newcommand{\force}{\Vdash}

\newcommand{\fullmiller}{\miller_\mathsf{full}}

\newcommand{\height}{\mathsf{ht}}

\newcommand{\level}{\textsc{Lev}}
\newcommand{\levy}{\mathsf{Coll}}

\newcommand{\miller}{\poset{M}}

\newcommand{\model}{\textsc}

\newcommand{\ns}{\textsc{NS}}
\newcommand{\On}{\text{On}}

\newcommand{\poset}{\mathbb}

\newcommand{\restric}{{\upharpoonright}}

\newcommand{\sacks}{\poset{S}}

\newcommand{\silver}{\poset{V}}

\newcommand{\SSigma}{\mathbf{\Sigma}}

\newcommand{\splitting}{\textsc{Split}}

\newcommand{\statsacks}{\sacks^\mathsf{stat}}
\newcommand{\statsilver}{\silver^\mathsf{stat}}
\newcommand{\statmiller}{\miller^\mathsf{stat}}
\newcommand{\stem}{\textsc{Stem}}
\newcommand{\successor}{\textsc{Succ}}

\newcommand{\term}{\textsc{Term}}

\begin{document}

\maketitle

\begin{abstract}  \footnote{\textbf{Acknowledgement.} The author is particularly indebted to Philipp Schlicht, since some of the ideas in this paper were inspirired by his seminar given in Freiburg during January 2013. Further, the author is grateful to Luca Motto Ros, for a detailed explanation of the differences and analogies between the Miller measurability and Hurewicz dichotomy, which can be found in Remark \ref{final-remark} of the present paper. Finally, the author thanks the referee for his/her useful suggestions.}
We present some results about the burgeoning research area concerning set theory of the ``$\kappa$-reals''. We focus on some notions of measurability coming from generalizations of Silver and Miller trees.
We present analogies and mostly differences from the classical setting.
\end{abstract}

\section{Introduction and basic definitions} \label{section:introduction}
The study of the generalized version of the Baire space $\kappa^\kappa$ and Cantor space $2^\kappa$, for $\kappa$ uncountable regular cardinal, is a burgeoning research area, which intersects both the generalized descriptive set theory and the set theory of the ``$\kappa$-reals'', where we refer to the elements of $\kappa^\kappa$ and $2^\kappa$ as $\kappa$-reals.

\textsf{Basic Notation.} The \emph{dramatis personae} of our work are the so-called tree-like forcings.
A tree $T$ is a subset of either $2^{<\kappa}$ or $\kappa^{<\kappa}$, closed under initial segments. $\stem(T)$ denotes the longest node of $T$ compatible with all the other nodes of $T$; $\successor(t,T):=\{\xi <\kappa: t^\conc \xi \in T \}$; \splitting(T) is the set of splitting nodes of $T$; we put $\height(T):= \sup \{\alpha: \exists t \in T (|t|=\alpha) \}$, while $\term(T)$ denotes the \enfa{terminal} nodes of $T$. For $\alpha < \kappa$,  $T\restric \alpha:= \{ t \in T: |t| < \alpha \}$. A \emph{branch} through $T$ is the limit of an increasing cofinal sequence $\{ t_\xi: \xi < \kappa \}$ of nodes in $T$, and $[T]$ will denote the set of all branches of $T$.
Moreover we will assume $\kappa^{<\kappa}=\kappa$ and that $\kappa$ is regular. Note that we will use the usual letters for denoting forcing notions like Sacks, Silver, Miller and Cohen, omitting the symbol $\kappa$, as it is clear that, in this paper, we will always deal with some generalized version.

Our attention will be particularly focused on the following types of trees:
\begin{itemize}
\item $T \subseteq 2^{<\kappa}$ is called \emph{Sacks} iff $\forall t \in T \exists t' \in T (t \subseteq t' \land t' \in \splitting(T))$ (we write $T \in \sacks$);
\item $T \subseteq 2^{<\kappa}$ is \emph{club Sacks} iff it is Sacks and for every $x \in [S]$ we have $\{ \alpha < \kappa: x \restric \alpha \in \splitting(T) \}$ is closed unbounded (we write $T \in \clubsacks$); analogously we define $\statsacks$ by requiring $\{ \alpha < \kappa: x \restric \alpha \in \splitting(T) \}$ to be stationary;
\item $T \subseteq 2^{<\kappa}$ is \emph{Silver} iff it is Sacks and moreover for every $s,t \in T$ such that $|s|=|t|$ one has $s^\conc i \ifif t^\conc i$, for $i \in \{ 0,1 \}$ (we write $T \in \silver$);
\item $T \subseteq 2^{<\kappa}$ is \emph{club Silver} iff it is Silver and $\level(T):= \{\alpha < \kappa: \exists t \in T (t \in \splitting(T))  \}$ is closed unbounded (we write $T \in \clubsilver$); analogously for $\statsilver$;
\item $T \subseteq \kappa^{<\kappa}$ is called \emph{Miller} iff $\forall t \in T \exists t' \in T (t \subseteq t' \land t' \in \splitting(T) \land |\successor(t,T)|=\kappa)$ (we write $T \in \miller$);
\item $T \subseteq \kappa^{<\kappa}$ is called \emph{club Miller} ($T \in \clubmiller$) iff it is Miller and the following hold:
\begin{itemize}
\item for every $x \in [T]$ one has $\{ \alpha < \kappa: x \restric \alpha \in \splitting(T) \}$ is closed unbounded,
\item for every $t \in \splitting(T)$ one has $\{ \alpha < \kappa: t^\conc \alpha \in T \}$ is closed unbounded.
\end{itemize}
\item $T \subseteq \kappa^{<\kappa}$ is called \emph{full Miller} ($T \in \fullmiller$) iff it is Miller and for every $t \in \splitting(T)$ for every $\alpha < \kappa$, one has  $t^\conc \alpha \in T $.
\end{itemize}
The associated forcing notions are ordered by inclusion.
We remark that a stronger version of $\clubmiller$ has been introduced by Friedman and Zdomskyy in \cite{FZ10}, where they proved that such a version, combined with club Sacks, preserves $\kappa^+$ via $<\kappa$-support iteration.

\

\noindent 

For our tree-forcings $\poset{P}$, one can introduce a corresponding notion of regularity as follows. 

\begin{definition}
A set $X$ of $\kappa$-reals is said to be:
\begin{itemize}
\item[-] $\poset{P}$-\enfa{measurable} iff 
\[
\forall T \in \poset{P} \exists T' \in \poset{P}, T' \subseteq T ([T'] \subseteq X \vee [T'] \cap X = \emptyset).
\]
\item[-] \enfa{weakly}-$\poset{P}$-\enfa{measurable} iff 
\[
\exists T \in \poset{P} ([T] \subseteq X \vee [T] \cap X = \emptyset).
\]
\end{itemize}
When $\poset{P}$ is one of our tree forcings, we may also use the notation ``Sacks measurable", ``Miller measurable" and so on. 
\end{definition}
In \cite{FKK14}, the authors show that the following generalization from the classical setting holds true: if $\Theta$ is a family of sets of $\kappa$-reals closed under intersection with closed sets and continuous pre-images, then
\[
\forall X \in \Theta (X \text{ is weakly-$\poset{P}$-measurable}) \Leftrightarrow  \forall X \in \Theta (X \text{ is $\poset{P}$-measurable}).  
\] 
Hence, when one investigates the validity of $\poset{P}$-measurability for all sets in $\Theta$, it is actually sufficient to investigate the weak-$\poset{P}$-measurability on $\Theta$.

\

There are essentially two main reasons for which the investigation of regularity properties in $\kappa^\kappa$ is interesting and more involved than the classical setting:
\begin{enumerate} 
\item the club filter is a $\SSigma^1_1$ set without the Baire property, as was proved by Halko and Shelah in \cite{HS01};
\item there is not an analogue of the factoring lemma for the Levy collapse $\levy(\kappa, <\lambda)$, for $\kappa> \omega$ and $\lambda$ inaccessible. More precisely, there are $x \in \kappa^\kappa$ such that $\levy(\kappa,<\lambda) / x$ is not equivalent to $\levy(\kappa,<\lambda)$.
\end{enumerate} 
We will call such reals \enfa{bad}, while on the opposite side, the \enfa{good reals} will be those having quotients equivalent to the Levy collapse.

Actually 2 is true even for the $\kappa$-Cohen forcing $\cohen$, as it is well-know that one can pick a Cohen $\kappa$-real $x$ and then a forcing $P$ shooting a club through the complement of $x$, and  this two step iteration is equivalent to $\cohen$. Hence, both $\cohen$ and $\levy(\kappa,<\lambda)$ are not strongly homogeneous, unlike their counterparts in the standard setting. 

Nevertheless, even if one cannot hope for a full factoring lemma, in \cite{Sc13} Philipp Schlicht has shown that one can recover a partial version. Indeed, he has proven that when forcing with $\levy(\kappa,<\lambda)$, one can obtain perfectly many good reals, in a sense, in order to use the usual Solovay's argument and obtain that all $\On^\kappa$-definable subsets of $2^\kappa$ have the perfect set property.  Inspired by his method, we prove some variants that will allow us to get the following two results:
\begin{itemize}
\item[($\star$)] for $\kappa$ inaccessible, a $\kappa^+$-iteration of $\cohen$ with $<\kappa$-support forces all $\On^\kappa$-definable sets to be  $\statsilver$-measurable;
\item[($\star \star$)] for $\lambda$ inaccessible, $\levy(\kappa, <\lambda)$ forces all $\On^\kappa$-definable sets to be $\fullmiller$- measurable;
\end{itemize}
Furthermore, we will also prove that ($\star$) is no longer true when one replaces $\statsilver$ with $\clubsilver$.  
We conclude this introductory section with a schema of the paper: in section \ref{section:preliminary} we show some interesting construction involving Sacks trees and Miller trees, marking some difference from the standard setting; in section \ref{perfect-cohen} we present some results concerning adding perfect trees of Cohen branches; in section \ref{silver} we build the model to get all $\On^\kappa$-definable subsets of $2^\kappa$ to be $\statsilver$-measurable; in section \ref{miller}, we prove that $\levy(\kappa, <\lambda)$ forces all $\On^\kappa$-definable subsets of $\kappa^\kappa$ to be $\fullmiller$-measurable; a concluding section is finally devoted to discussing some further potential developments.

\section{Some basic differences from the classical setting} \label{section:preliminary}
This section may be read independently from the rest of the paper. It is devoted to analyzing some basic differences from the standard setting. Throughout this section, we assume that $\kappa$ is a regular successor. Let $\Gamma := \{\lambda_\alpha:\alpha<\kappa\}$ be such that $\lambda_0=0$ and $\{\lambda_\alpha: 1 \leq \alpha < \kappa \}$ enumerates the limit ordinals $<\kappa$ such that $2^{\lambda_\alpha}=\kappa$. For $t \in 2^{\lambda_\alpha}$ and $\alpha<\kappa$, let
$\pi(t,\alpha+1):= \{t' \in 2^{\lambda_{\alpha+1}}: t' \supseteq t  \}$.
Furthermore, fix a well-ordering $W(t,\alpha+1)= \{ t^{\alpha+1}_\xi: \xi < \kappa   \}$ of $\pi(t,\alpha+1)$.

In the standard case when $\kappa=\omega$, we know that $2^\omega$ and $\omega^\omega$ are not homeomorphic, even if they are connected via a Borel isomorphism. The following simple remark shows that the situation is different, when $\kappa> \omega$ is a successor. The following result was proved in \cite{HN73}. We give a sketch of the proof, since the construction is needed later. 
\begin{remark}
$2^\kappa$ and $\kappa^\kappa$ are homeomorphic. Moreover, there are \emph{many} such homeomorphisms.
Indeed, let $f:\kappa^{<\kappa} \rightarrow 2^{<\kappa}$ be defined recursively as follows:
\begin{itemize}
\item[-] $f(\emptyset)=\emptyset$
\item[-] $f(\langle \xi \rangle)$ is the $\xi$th element of the well-ordering $W(\emptyset,1)$
\item[-] if $t \in \kappa^{\alpha}$ and $\xi<\kappa$, $f(t^\conc\xi)$ is the $\xi$th element in the well-ordering $W(f(t),\alpha+1)$.
\item[-] given $\{t_\alpha: \alpha <\gamma\}$ increasing sequence, $\gamma$ limit ordinal, put $f(\lim_{\alpha<\gamma}t_\alpha)=\lim_{\alpha<\gamma} f(t_\alpha)$.
\end{itemize}
Notice that the range of $f$ is not $2^{<\kappa}$, but it is a strict subset of it, namely $\bigcup \{t \in 2^{\lambda_\alpha}:\lambda_\alpha \in \Gamma \}$. This $f$ provides a bijection $h:\kappa^\kappa \rightarrow 2^\kappa$ in the natural way, that is $h(x)=\lim_{\alpha<\kappa}f(x \restric \alpha)$, for every $x \in \kappa^\kappa$. It easily follows from the definition that $h$ is a homeomorphism.
\end{remark}

The homeomorphism obviously depends on the well-orderings of $W(t,\alpha)$, and so it is far from being uniquely determined.
We now want to use these homeomorphisms between $\kappa^\kappa$ and $2^\kappa$ to exhibit some particular situations which do not occur in the standard setting.

\begin{fact} \label{fact}
For every club Miller tree $T \subseteq \kappa^{<\kappa}$ with the property that for every $t \in T$, $|\{\xi: t^\conc \xi \notin T  \}|=\kappa$, there exists a homeomorphism $h$ such that $h''[T]$ does not contain the branches of a club Sacks tree. 
\end{fact}
\begin{proof}
Let $T \subseteq \kappa^{<\kappa}$ be a club Miller tree. Instead of $h$, we actually define the function $f:\kappa^{<\kappa} \rightarrow 2^{<\kappa}$, from which we will naturally obtain the desired $h$. For every club splitting node $t \in T$ we define $f$ satisfying the following requirement: let $C_t \subseteq \kappa $ denote the club set of successors of $t$,  then
$\xi \in C_t \Rightarrow f(t^\conc \xi) \supseteq f(t)^\conc 1$.  
It is then clear that, for every limit ordinal $\alpha$ and every $t \in f''T$ with length $\lambda_\alpha$, we get that $f(t)$ cannot be a splitting node. Hence, $f''T$ cannot contain a club Sacks subtree.

\end{proof}
\begin{lemma} \label{bad lemma}
There exist a homeomorphism $h: \kappa^\kappa \rightarrow 2^\kappa$ and $Y \subseteq \kappa^{\kappa}$ such that $Y$ is weakly club Miller measurable but $h''Y$ is not weakly club Sacks measurable. 
\end{lemma}
\begin{proof}
Consider $f$, $h$ and $T$ as in Fact \ref{fact}. Note that, for every club Sacks tree $S$, $[S] \setminus h''[T]$ has cardinality $2^\kappa$. This follows easily from Fact \ref{fact}, since there is $\alpha \in \kappa$ and $t \in 2^{\lambda_\alpha}$ such that $t^\conc 0 \in S \land f^{-1}(t^\conc 0) \notin T$ (actually there are cofinally many such $\alpha$'s). 

Let $\{S_\xi : \xi < 2^\kappa \}$ be an enumeration of all club Sacks trees.
Now we construct $\{ Y_\xi: \xi < 2^\kappa \}$ and $\{ Z_\xi: \xi < 2^\kappa \}$ recursively as follows:
\begin{itemize}
\item[-] \textsf{Step} $-1$: $Y_0=[T]$ and $Z_0=\emptyset$;
\item[-] \textsf{Step} $\xi$ successor or $\xi=0$: pick $y_{\xi} \in {h^{-1}}[S_{\xi}] \setminus \bigcup_{\iota \leq \xi} Y_\iota$ and put $Y_{\xi+1}=Y_{\xi} \cup \{ y_{\xi} \}$. Then pick $z_{\xi} \in {h^{-1}}[S_{\xi}] \setminus \bigcup_{\iota \leq \xi} Z_{\iota}$ such that $z_{\xi} \notin Y_{\xi+1}$, and put $Z_{\xi+1}=Z_{\xi} \cup \{ z_{\xi} \}$. Note that the choice of $y_\xi$ can be done, since by Fact \ref{fact} any club Sacks set contains $2^\kappa$ many branches which are not in $h''[T]$.  
\item[-] \textsf{Step} $\xi$ limit: put $Y_\xi= \bigcup_{\iota < \xi} Y_\iota$ and $Z_\xi= \bigcup_{\iota < \xi} Z_\iota$.
\end{itemize}
Finally put $Y= \bigcup_{\xi < 2^\kappa} Y_\xi$ and $Z= \bigcup_{\xi < 2^\kappa} Z_\xi$. Then for all club Sacks tree $S$ both $h''Y \cap [S] \neq \emptyset$ and $[S] \nsubseteq h''Y$ (the latter because $Z \cap Y = \emptyset$.)
\end{proof}

On the opposite side, we have the following.
\begin{lemma} \label{good lemma}
Assume $f: \kappa^{<\kappa} \rightarrow 2^{<\kappa}$ is a map as above and satisfying the following further property: for every $\alpha < \kappa$ and every $t \in 2^{\lambda_\alpha}$, 
\[
\emph{($\dagger$)} \quad {f^{-1}}\{ t' \in 2^{\lambda_{\alpha+1}}: t^\conc 0 \subset t'  \}, {f^{-1}}\{ t' \in 2^{\lambda_{\alpha+1}}: t^\conc 1 \subset t'  \} \text{ are stationary}.
\]
Then for every club Miller tree $T$ we have that $f'' T$ contains a club Sacks tree. 
\end{lemma}
\begin{proof}
Indeed, we are going to prove the following stronger conclusion: let $\sacks^{*,\textsf{club}}$ be the version of club Sacks forcing obtained by replacing $2^{<\kappa}$ with $\kappa^{<\kappa}$, i.e., $\sacks^{*,\textsf{club}}$ consists of 2-branching trees in $\kappa^{<\kappa}$ with club splitting. Then we are going to prove that for every $T \in \clubmiller$ there exists $T_0 \subset T$ in $\sacks^{*,\textsf{club}}$ and $S \in \clubsacks$ such that $f''T_0 = S$. We recursively construct $T_0$ as follows.

\textsf{Step 0.} Let $t_\emptyset= \stem(T)$ and put $\alpha_\emptyset = |t_\emptyset|$. Then pick $\xi_0, \xi_1 \in \successor(t_\emptyset)$ such that $f({t_\emptyset}^\conc \xi_0) \supset f(t_\emptyset)^\conc 0$ and $f({t_\emptyset}^\conc \xi_1) \supset f(t_\emptyset)^\conc 1$; note that this can be done by ($\dagger$), since $\successor(t_\emptyset, T)$ is club. Further, for $i \in \{ 0,1 \}$, let $t_{\langle i \rangle}$ be the least splitting node in $T_{t_\emptyset^\conc i}:= \{s \in T: s \subseteq {t_\emptyset}^\conc i \vee s \supseteq {t_\emptyset}^\conc i  \}$. 

\textsf{Successor step.}  Assume the construction done for all $\sigma \in 2^{\beta}$. For every $\sigma \in 2^\beta$, we use the same idea as step 0, and we pick $\xi_0, \xi_1 \in \successor({t_{\sigma}})$ such that, for $i \in \{ 0,1 \}$, $f({t_\sigma}^\conc \xi_i) \supset f(t_\sigma)^\conc i$. Analogously, $t_{\sigma^\conc i}$ is the least splitting node in $T_{t_\sigma^\conc i}$.

\textsf{Limit step.} For $\sigma \in 2^{\delta}$ such that, for all $\beta < \delta$, $t_{\sigma \restric \beta}$ is already constructed, put $t_{\sigma}:= \bigcup_{\beta < \delta} t_{\sigma \restric \beta}$.

Finally let $T_0$ be the downward closure of $\bigcup_{\sigma \in 2^{<\kappa}} t_{\sigma}$ and $S:= f''T_0$. By construction, $T_0$ and $S$ have the required properties.
\end{proof}
\begin{remark}
In a sense, the situation occurring in Lemma \ref{bad lemma} is very unpleasant, as we would like to generally view Miller trees as particular kind of Sacks trees, and moreover that this fact is preserved under homeomorphism. Hence, Lemma \ref{bad lemma} and \ref{good lemma} may be understood as a way of separating good homeomorphisms from bad ones.
\end{remark}

\section{Generic trees of Cohen branches} \label{perfect-cohen}
We present some results about adding certain types of generic perfect trees. In section \ref{silver} and \ref{miller}, it will be crucial to use specific kinds of perfect trees such that each of their branches is Cohen over the ground model. We refer to such generic trees by saying ``perfect trees of Cohen branches".

\begin{lemma} \label{lemma:generic-silver}
Let $\kappa$ be inaccessible. Let $\mathbb{VT}:= \{ p: \exists T \in \silver \exists \alpha \in \kappa (p= T \restric \alpha)  \}$, ordered by end-extension, i.e., $p' \leq p$ iff $p \subseteq p' \land \forall t \in p' \setminus p \exists s \in \term(p)(s \subseteq t)$. Let $T_G:= \bigcup \{p: p \in G \}$, with $G$ being $\mathbb{VT}$-generic filter over the ground model $\model{N}$. Then
\[
\begin{split}
\model{N}[G]  \models & \quad  T_G \in \silver \land \forall x \in [T_G] (x \text{ is Cohen over } \model{N}) \\
			 &	\quad \land \level(T_G) \text{ is stationary and co-stationary},
\end{split}
\]
where $\level(T_G)$ denotes the set of splitting levels of $T_G$.
Moreover, $\mathbb{VT}$ is a forcing of size $\kappa$ and is $< \kappa$-closed. So it is actually equivalent to $\kappa$-Cohen forcing. 
\end{lemma}
\begin{proof}
Fix $p \in \mathbb{VT}$ and $D \subseteq \cohen$ open dense and let $\{ t_\alpha: \alpha < \delta < \kappa \}$, enumerate all terminal nodes of $p$ (w.l.o.g. assume $\delta$ is a limit ordinal).  We use the following notation: for every $s,t \in 2^{<\kappa}$, put
\[
t \oplus s := \{t' \in 2^{<\kappa}: \forall \alpha < |t|(t'(\alpha)=t(\alpha)) \wedge \forall \alpha \geq |t|( t'(\alpha)= s(\alpha))\}.
\]
Then consider the following recursive construction:
\begin{itemize}
\item[-] pick $s_0 \supseteq t_0$ such that $s_0 \in D$;
\item[-] for $\alpha+1$, pick $s_{\alpha+1} \supseteq t_{\alpha+1} \oplus s_\alpha$ such that $s_{\alpha+1} \in D$.
\item[-] for $\alpha$ limit, put $s'_\alpha= \bigcup_{\xi < \alpha} s_\xi$ and pick $s_{\alpha} \supseteq t_{\alpha} \oplus s'_\alpha$ such that $s_{\alpha} \in D$.
\item[-] once the procedure has been done for every $\alpha < \delta$, we put $s_\delta:= \bigcup_{\alpha<\delta} t_0 \oplus s_\alpha$ and then $t'_{\alpha}:= {t_\alpha} \oplus s_\delta$.
\end{itemize}
Note that to make sure that ${t_\alpha} \oplus s_\delta \in 2^{<\kappa}$ we need to use the assumption that $\kappa$ is inaccessible. 
Finally, let $p'$ be the downward closure of $\bigcup_{\alpha < \delta} t'_\alpha$. By construction, $p' \in \poset{VT}$, $p' \leq p$ and 
for every terminal node $t \in p'$, we get $t \in D$. Hence $p' \force \forall x \in [T_G](H_x \cap D \neq \emptyset)$, where $H_x:= \{ s \in \cohen: s \subset x \}$.  

We now want to further extend $p'$ in order to catch the second property as well, i.e., $\level(T_G)$ is both stationary and co-stationary. So fix $\dot C$ name for a club of $\kappa$. Build  sequences $\{ q_n: n \in \omega \}$ and $\{ \xi_n: n \in \omega \}$ such that: $q_0 = p'$, and $q_{n+1} \leq q_n$ such that $q_{n+1} \force \xi_n \in \dot C$ and $\xi_n > \height(q_{n})$ and $\height(q_{n+1}) > \xi_n$. Finally put $\xi_\omega= \lim_{n < \omega} \xi_n$, $q_\omega := \bigcup_{n \in \omega} q_n$, and then 
\[
p^*:= q_\omega \cup \bigcup \{ t^\conc i: t \in \term(q_\omega) \land i \in \{ 0,1 \}  \}. 
\]
Hence $p^* \force \forall n (\xi_n \in \dot C)$, and then $p^* \force \xi_\omega \in \dot C$. But $\xi_\omega = \height(q_\omega)$, since the $\xi_n$'s and the $|\height(q_n)|$'s are mutually cofinal, and hence $p^* \force \xi_\omega \in \level(T_G) \cap \dot C$. This shows that $\level(T_G)$ is stationary. For proving that it is co-stationary as well, we can further extend $p^*$, by using the same procedure, in order to find 
$\{ q'_n: n \in \omega \}$ and $\{ \xi'_n: n \in \omega \}$ as above and then
$p^{**} \leq q'_\omega$ such that $p^{**}:= q'_\omega \cup \bigcup \{ t^\conc 0: t \in \term(q'_\omega) \}$. Hence 
\[
p^{**} \force \xi_\omega \in \level(T_G) \cap \dot C \land \xi'_\omega \notin \level(T_G) \cap \dot C,
\] 
which completes the proof.
\end{proof}

About generic Miller trees of Cohen branches the situation is very different, since the above argument does not seem to work. The next method shows a simple different way to add a tree $T \in \fullmiller$ of Cohen branches, which we will use in section \ref{miller}. On the opposite side, Lemma \ref{lemma2:cohen-non-miller} marks a necessary condition for adding a tree $T \in \clubmiller$ of Cohen branches, generalizing some results obtained by Spinas and Brendle in the classical setting (see \cite{Sp95} and \cite{Br99}).

We use the following notation: given a tree $T \subseteq \kappa^{<\kappa}$,
\begin{itemize} 
\item[-] $\splitting_\alpha(T)$ is recursively defined as:  

$\splitting_0(T)=\{ \stem(T) \}$; 

$t \in \splitting_\alpha(T)$ iff $t \in \splitting(T)$ and for every $\beta < \alpha$ there exists $t_\beta \subset t$ such that $t_\beta \in \splitting_\beta(T)$.  
\item[-] $T[\alpha]:= \{s \in T: \exists t \in \splitting_\alpha(T) \exists i < \kappa (t^\conc i \in T \land s \subseteq t^\conc i)  \}$. 

\end{itemize}

\begin{lemma} \label{lemma:generic-miller}Define the forcing $\poset{MT}:= \{ p: \exists T \in \fullmiller \exists \alpha < \kappa(p \sqsupseteq T[\alpha]) \}$, ordered by end-extension.
Then $\poset{MT}$ adds a full Miller tree of Cohen branches.
\end{lemma}
\begin{proof}
Let $D \subseteq \cohen$ be open dense and $p \in \poset{MT}$. Pick $\phi: \term(p) \rightarrow \kappa^{<\kappa}$ such that $\phi(t) \in D$ and $\phi(t) \supseteq t$, and then define $p' \leq p$ as the downward closure of 
$\bigcup \{ \phi(t)^\conc \xi : t \in \term(p) \land \xi \in \kappa \}$. Then $p' \force \forall x \in [T_G](H_x \cap D \neq \emptyset)$, where $H_x:= \{ s \in \cohen: s \subset x \}$.
\end{proof}

\begin{remark} \label{true-amoeba}
Note that in the proof of both Lemma \ref{lemma:generic-silver} and Lemma \ref{lemma:generic-miller}, we have proved that one can add a certain type of generic tree whose branches are Cohen in the extension $\model{N}[G]$, where $G$ is $\poset{VT}$- and $\poset{MT}$-generic over $\model{N}$, respectively. In the application that we will see in the next sections, we actually need something stronger, i.e., that all branches of the generic tree have to be Cohen in \emph{any} extension $\model{M} \supseteq \model{N}[G]$ via a $<\kappa$-closed forcing. But, this is actually implicit in our proof. Indeed, in Lemma \ref{lemma:generic-miller} we have proven that, for every $D \subseteq \cohen$ open dense in $\model{N}$,  
\[
\model{N}[G] \models \varphi : \equiv \exists F \subseteq \kappa^{<\kappa} \forall x \in \kappa^\kappa (x \in [T_G] \Rightarrow \exists t \in F(t \subset x \land t \in D)), 
\]
and analogously for Lemma \ref{lemma:generic-silver} with $2^\kappa$ in place of $\kappa^\kappa$. 
\noindent
Note that this formula $\varphi$ is $\SSigma^1_2(\kappa^\kappa)$. Hence, it is upward absolute between $\model{N}[G]$ and any extension $\model{M}$ via $<\kappa$-closed forcing (this to ensure   $(\kappa^{<\kappa})^\model{M} = (\kappa^{<\kappa})^{\model{N}[G]}$ and then $\SSigma^1_1$-absoluteness). Hence, we get $\model{M} \models \varphi$, which means $\model{M} \models \forall x \in [T_G] (H_x \cap D \neq \emptyset)$. Since $D \in \model{N}$ was arbitrarily chosen, we have obtained: for every $D \subseteq \cohen \cap \model{N}$, for every $x \in [T_G]^\model{M}$, one has $H_x \cap D \neq \emptyset$. Hence,
$\model{M} \models \forall x \in [T_G](x \text{ is Cohen over } N)$. 
\end{remark}

\begin{lemma} \label{lemma2:cohen-non-miller}
Let $\model{M}$ be a ZFC-model extending the ground model $\model{N}$. If for all $x \in \kappa^\kappa \cap \model{M}$ there exists $y \in \kappa^\kappa \cap \model{N}$ such that $\forall \alpha < \kappa \exists \beta \geq \alpha (x(\beta) < y(\beta))$, then in $\model{M}$ there is no club Miller tree of Cohen branches. In other words, If one adds a club Miller tree of Cohen branches, then one necessarily adds dominating $\kappa$-reals over the ground model.
\end{lemma}

\begin{proof}
Let $T \in \clubmiller$ and $t \in \splitting(T)$. Define 
\[
h_t(\alpha):= \min \{ |t'|: \exists \xi \geq \alpha (t' \in \splitting(T) \land t' \supseteq t^\conc \xi)  \} + 1.
\]
Further, given $z \in \kappa^{\uparrow \kappa} \cap \model{N}$, define $B(z):= \{ x \in \kappa^\kappa: \forall \mu \forall \alpha \leq \mu (z(x(\alpha)) \geq \mu) \}$.
\begin{claim}
$B(z)$ is closed nowhere dense.
\end{claim}
\begin{proof}[Proof of Claim]
To see that $B(z)$ is nowhere dense, fix $s \in \kappa^{<\kappa}$. Then let $s'= s^\conc 0_\beta$, where $0_\beta$ is the sequence of $0$s of length $\beta$, and $\beta$ is sufficiently large that $|s'|> \sup\{ z(s(\alpha)): \alpha < |s| \}$. Hence $[s'] \cap B(z)=\emptyset$.  
\end{proof}

Let $T \in \clubmiller \cap \model{M}$ be a tree of Cohen branches over $\model{N}$. Pick $h \in \kappa^\kappa$ such that $\forall t \in \splitting(T) \exists \alpha < \kappa \forall \xi \geq \alpha (h_t(\xi) < h(\xi))$. To show that $h$ is dominating over $\model{N}$, we argue by contradiction; pick $z \in \kappa^{\uparrow \kappa} \cap \model{N}$ which is not eventually dominated by $h$, and with the further property that $z(0)>|\stem(T)|$.  Let us construct $\{ t_\xi: \xi < \kappa \}$ recursively as follows:
\begin{itemize}
\item $t_0= \stem(T)$ and for $\lambda$ limit ordinal let $t_\lambda= \bigcup_{\xi < \lambda} t_\xi$.
\item Assume $t_\xi$ already defined. By the choice of $z$, there exists $\beta \in \kappa$ such that $h(\beta) < z(\beta)$. We distinguish two cases:
\begin{itemize}
\item if ${t_\xi}^\conc \beta \in T$, then simply put $t_{\xi+1}$ be the least splitting node extending ${t_\xi}^\conc \beta$;
\item if ${t_\xi}^\conc \beta \notin T$, then let $\gamma_\xi:= \min \{ \gamma: \gamma > \beta \land {t_\xi}^\conc\gamma \in T  \}$.  By construction, $h(\gamma_\xi)=h(\beta)$ and so $h(\gamma_\xi) < z(\beta)\leq z(\gamma_\xi)$, since $z$ is increasing. Then let $t_{\xi+1}$ be the least splitting node of $T$ extending ${t_\xi}^\conc \gamma_\xi$. 
\end{itemize}
Note that when $\xi$ is limit, by recursive construction, $t_\xi \in \splitting(T)$, as $t_\xi$ is a limit of splitting nodes in $T$.  Hence the construction works even for $\xi$ successor of a limit ordinal.
Finally let $x= \bigcup_{\xi < \kappa} t_\xi$. It is left to show that $x \in [T] \cap B(z)$, which will give us $x \in [T]$ not Cohen over $\model{N}$, since $B(z) \in \model{N}$ is nowhere dense. Clearly $x \in [T]$, since the construction explicitely gives us cofinally many $\alpha < \kappa$ such that $x \restric \alpha \in T$. To show that $x \in B(z)$, we argue as follows: for every $\alpha<\kappa$, pick the least $\xi<\kappa$ such that $\alpha < |t_\xi|$. By induction over $\xi < \kappa$:
\begin{itemize}
\item $\xi=0$: for every $\alpha < |\stem(T)|$, we have $z(x(\alpha)) > |\stem(T)|$;
\item $\xi$ limit ordinal: trivial;
\item $\xi+1$:  if $\alpha < |t_\xi|$ use inductive hypothesis. If $|t_\xi|\leq \alpha < |t_{\xi+1}|$, then $x(|t_\xi|)=t_{\xi+1}(|t_\xi|)=\gamma_\xi$, and so by choice of $\gamma_\xi$, it follows that for every $\alpha < |t_{\xi+1}|$, $z(x(\alpha))\geq z(\gamma_\xi) > |t_{\xi+1}|$, since $z$ is increasing.
\end{itemize} 
\end{itemize}
\end{proof}

\begin{corollary}
$\cohen$ does not add a generic $T \in \clubmiller$ of Cohen branches.
\end{corollary}

\section{Stationary-Silver vs club-Silver} \label{silver}
Silver forcing may be introduced by using partial functions $f:\kappa \rightarrow 2$, ordered by extension; simply identify such an $f$ with the tree $T_f := \{x \in 2^\kappa: \forall \alpha \in \dom(f) (f(\alpha)=x(\alpha))  \}$. We will use $T$ and $f_T$ interchangeably, depending on the situation. Throughout this section, $T_f$ will denote the tree associated with a given $f$, and vice versa, $f_T$ will denote the partial function associated with a given $T$. Note that $\dom(f)= \kappa \setminus \level(T_G)$.

In this section we want to investigate the family of $\clubsilver$-measurable and $\statsilver$-measurable sets. 

\begin{lemma} \label{lemma:non-club-silver-meas}
There exists a $\SSigma_1^1$ set which is not $\clubsilver$-measurable (i.e., the club filter \club).
\end{lemma}
\begin{lemma} \label{lemma:projective-silver-meas}
Assume $\kappa$ be inaccessible. Let $\cohen_{\kappa^+}$ be a $\kappa^+$-iteration of $\kappa$-Cohen forcing with $< \kappa$ support, and let $G$ be the $\cohen_{\kappa^+}$-generic filter over $\model{N}$. Then 
\[
\model{N}[G] \models \text{`` all \emph{On}$^\kappa$-definable sets in $2^\kappa$ are $\statsilver$-measurable. ''}
\]
\end{lemma}
\textsf{Notation}: we will abuse notation by saying that ``$x \in 2^\kappa$ is in \club'', instead of the more correct ``$\{ \alpha < \kappa: x(\alpha)=1 \}$ is in \club''.

We start with the proof of the easier of the two lemmata.
\begin{proof}[Proof of Lemma \ref{lemma:non-club-silver-meas}.] We will show that for every $T \in \clubsilver$, 
\[
\exists x \in 2^\kappa (x \in \club \cap [T]) \land \exists y \in 2^\kappa (y \in \ns \cap [T]),
\]
where $\ns$ is the ideal of non-stationary subsets of $\kappa$. Define $x \in 2^\kappa$ as follows:
\[
x(\alpha):=  
\begin{cases}
f_T(\alpha)		&\text{ if $\alpha \in \dom(f_T)$}, \\
1 				&\text{ otherwise}.
\end{cases}
\]
Then obviously $x \supseteq \level(T)$ and so $x \in \club \cap [T]$. Analogously, we can define
\[
y(\alpha):=  
\begin{cases}
f_T(\alpha)		&\text{ if $\alpha \in \dom(f_T)$}, \\
0				&\text{ otherwise}.
\end{cases}
\]
Hence, $y \in \ns \cap [T]$.
\end{proof}
The rest of this section is devoted to prove Lemma \ref{lemma:projective-silver-meas}. We use a variant of Schlicht's method to only work with branches having good quotient. 
We need the following key lemma. Hereafter, $\poset{VT}_\alpha$ denotes the $<\kappa$-support $\alpha$-iteration of $\poset{VT}$, introduced in section \ref{perfect-cohen}.

\begin{lemma} \label{quotient-lemma}
Let $\alpha < \kappa^+$. Let $\dot T$ be the canonical $\poset{VT}_0$-name for the generic Silver tree added by $\poset{VT}_0$, and $\dot x$ be a $\poset{VT}_\alpha$-name for a Cohen branch through $\dot T$. Let $G$ be the $\poset{VT}_\alpha$-generic filter over $\model{N}$ and $z=\dot x^G$. Then $\poset{VT}_\alpha / \dot x{=}z$ is equivalent to $\poset{VT}_\alpha$.
\end{lemma}
Note that, unlike Schlicht's work, here the name for a branch comes from a ``larger'' forcing than the one adding the generic tree. So we need a slight generalization of his argument. 

\textsf{Notation}: from now on, $\dot x$, $\dot T$, $G$ will be as in the statement of Lemma \ref{quotient-lemma}, while $x_p$ will denote the initial segment of $\dot x$ decided by $p:= \langle \dot p(\xi):\xi < \alpha \rangle \in \poset{VT}_\alpha$.

We prove some preliminary results. 

\begin{claim} \label{silver-claim1}
$\poset{VT}^*_\alpha:=\{ p \in \poset{VT}_\alpha: |x_p| \geq \height(p(0)) \}$ is dense in $\poset{VT}_\alpha$.
\end{claim}
\begin{proof}
Given $p \in \poset{VT}_\alpha$ we have to find $p' \leq p$ in $\poset{VT}^*_\alpha$.
Start with $p_0:=p$ and then, for every $n \in \omega$, pick $p_{n+1} \leq p_n$ such that $|x_{p_{n+1}}| > \height(p_{n}(0))$. Let $p_\omega:=\bigcup_{n \in \omega} p_n$ and $w:=\bigcup_{n \in \omega} x_{p_n}$. Then $w \subseteq x_{p_\omega}$ and $|w|=\height(p_\omega(0))$. Hence $p':=p_\omega$ has the required property. 
\end{proof} 

In the following two claims, we need to work with conditions forcing $\dot x \in \dot T$. Note that, for every $p_0 \in \poset{VT}^*_\alpha$ we can always find  $p \leq p_0$ such that $p \force \dot x \in \dot T$. Hence, from now on, we will always consider conditions $p$ sufficiently strong to force $\dot x \in \dot T$. 
\begin{claim} \label{silver-claim2}
For every $p \in \poset{VT}^*_\alpha$ we have $|x_p|=\height(p(0))$.
\end{claim}
\begin{proof} 
Note that $p \force \dot x \in \dot T \land p(0) \sqsubset \dot T$, where $\sqsubset$ means ``initial segment"; hence, there exists $t \in \term(p(0))$ such that $p \force t \subset \dot x$. By contradiction, assume $x_p=t^\conc s$, for some $t \in \term(p(0))$ and non-empty $s \in 2^{<\kappa}$. Let $S$ be the downward closure of $\bigcup\{t \oplus t': t \in \term(p(0))\}$, for some $t' \perp  t^\conc s$ with $t' \supset t$. Let $p' \in \poset{VT}_\alpha$ be defined as 
\[
p' (\iota):=  
\begin{cases}
S		&\text{ if $\iota=0$}, \\
\dot p(\iota)			&\text{ if $\iota > 0$}.
\end{cases}
\] 
Then pick $p^* \leq p'$ such that $p^* \in \poset{VT}^*_\alpha$. Since $p^* \force S \sqsubset \dot T$ and $|x_{p^*}| \geq \height(S)$, it follows that  $t' \subseteq x_{p^*}$ and so $x_{p^*} \perp x_p$, contradicting $p^* \leq p$.
\end{proof}

\begin{claim} \label{silver-claim3}
\[
\forall p \in \poset{VT}^*_\alpha \forall s \in 2^{<\kappa}(x_p \subseteq s \Rightarrow  \exists p^* \in \poset{VT}^*_\alpha (s \subseteq x_{p^*})).
\]
\end{claim}
\begin{proof}
The argument is very similar to the one above. 
Note that for every $p \in \poset{VT}^*_\alpha$, there exists $t_0 \in \term(p(0))$ such that $t_0 = x_p$.
Pick $s \in 2^{<\kappa}$ such that $x_p \subseteq s$.
Let $S$ be the downward closure of $\bigcup\{ t \oplus s: t \in \term(p(0)) \}$.
Define $p' \in \poset{VT}_\alpha$ as follows :
\[
p'(\iota):=  
\begin{cases}
S		&\text{ if $\iota=0$}, \\
\dot p(\iota)			&\text{ if $\iota > 0$}.
\end{cases}
\]
Then pick $p^* \in \poset{VT}^*_\alpha$ such that $p^* \leq p'$. Since $p^* \force S \sqsubset \dot T$ and $|x_{p^*}| \geq \height(S)$, it follows that $s \subseteq x_{p'} \subseteq x_{p^*}$.
\end{proof}
\begin{corollary} \label{cor:dense-cohen}
Let $D \subseteq \poset{VT}^*_\alpha$ be open dense. Then $W_q:=\{ x_p \in 2^{<\kappa}: p \in D \land p \leq q \}$ is dense in $\cohen$ below $x_q$.
\end{corollary}
\begin{proof}[Proof of Lemma \ref{quotient-lemma}]
 The proof is completely analogous to the one of Schlicht for $\cohen$. We give it for completeness and because we actually deal with $\poset{VT}_\alpha$-names for branches in $\dot T$ instead of $\poset{VT}$-names only.  

We will prove the lemma for $\poset{VT}^*_\alpha$, but since it is forcing equivalent to $\poset{VT}_\alpha$, the same will hold true for the latter as well (and then even for $\cohen_\alpha$). It is well-known that $\poset{VT}^*_\alpha / \dot x{=}z = \poset{VT}^*_\alpha \setminus \bigcup_{\beta < \gamma} A_\beta$, where the elements of this union are recursively defined in $\model{N}[z]$ as follows:
\[
\begin{split}
A_0 		&:=  \{ p \in \poset{VT}^*_\alpha: \exists \xi < \kappa (p \force \dot x(\xi)\neq z(\xi)) \}. \\
A_{\beta+1} &:=  \{ p \in \poset{VT}^*_\alpha: \exists D \subseteq A_\beta \text{ open dense below $p$ }, D \in \model{N} \}.\\
A_\lambda 	&:=  \bigcup_{\beta < \lambda} A_\beta, \text{ for $\lambda$ limit ordinal},\\
\end{split}
\]
and finally $\gamma$ is chosen so that $A_\gamma=A_{\gamma+1}$. 

Note that $\gamma=0$; by contradiction, pick $p \in A_1 \setminus A_0$. Since $p \in A_1$, it follows that there exists $D \subseteq A_0$ such that $D \in \model{N}$ and $D$ is dense below $p$. Then the set $W_p:=\{ x_{p'} \in 2^{<\kappa}: p' \in D \land p' \leq p \}$ is dense in $\cohen$ below $x_p$, by Corollary \ref{cor:dense-cohen}, and so there exists $p' \in D$ such that $x_{p'} \subset z$, as $z$ is Cohen over $\model{N}$ (and $x_p \subset z$, by $p \notin A_0$). Also since $D \subseteq A_0$, it follows $p' \in A_0$. But, by definition, 
\[
\begin{split}
p' \in A_0 		&\Leftrightarrow p' \force \dot x(\xi)\neq z(\xi), \text{ for some $\xi < \kappa$}\\
				&\Leftrightarrow x_{p'}  \not \subset z, 
\end{split}
\]
providing us with a contradiction. Hence we get
\[
\poset{VT}^*_\alpha / \dot x{=}z = \{ p \in \poset{VT}^*_\alpha: \forall \xi < \kappa (p \not \force \dot x(\xi) \neq z(\xi)) \}= \{ p \in \poset{VT}^*_\alpha: x_p \subset z \},
\]
which is a $<\kappa$-closed subset of a forcing equivalent to $\cohen$, and so it is in turn equivalent to $\cohen$.
\end{proof}
We now have all tools needed for proving the main lemma of this section.
\begin{proof}[Proof of Lemma \ref{lemma:projective-silver-meas}]
Let $X \subseteq 2^\kappa$ be a set defined by some formula $\varphi$ with ordinal parameters and $v \in \On^\kappa$, which we may assume to be absorbed into the ground model, by the $\kappa^+$-cc. Also, for any $x \in [T_{G_0}]^{\model{N}[G]}$, there is $\alpha < \kappa^+$ and a $\cohen_\alpha$-name $\dot x$ for such $x$. Note that, by Remark \ref{true-amoeba}, $x$ is Cohen over $\model{N}$, and by Lemma \ref{quotient-lemma}, $\dot x$ has good quotient in $\cohen_{\alpha}$, and hence in $\cohen_{\kappa^+}$ as well. Indeed, $\cohen_{\kappa^+}$ can be viewed as $Q_{\dot x} * \dot R_0 * \dot R_1 $, where $Q_{\dot x}$ is the forcing generated by $\dot x$ (and so it is equivalent to $\cohen$ as $x$ is Cohen over $\model{N}$), while $\force_{Q_{\dot x}} \dot R_0 \cong \cohen_{\alpha}$ (that means, $\model{N}[x] \models \dot R_0^x \cong \cohen_{\alpha}$), since $x$ has good quotient, and finally $\dot R_1$ is just a ``tail'' of $\cohen_{\kappa^+}$, and so it is equivalent to $\cohen_{\kappa^+}$ itself. So let us put $\dot R= \dot R_0 * \dot R_1$, so to have  $\model{N}[x] \models \dot R^x \cong \cohen_{\kappa^+}$.
  
Let $x$ be Cohen over $\model{N}$ with good quotient. Then 
\[
\model{N}[x] \models \text{``}\force_{\dot R^x} \varphi(x)\text{''} \quad \text{or} \quad \model{N}[x] \models \text{``} \not \force_{\dot R^x} \varphi(x)\text{''}.
\]
Assume the former, and put $\theta(x):= \text{``}\force_{\dot R^x} \varphi(x)\text{''} $ Then there exists $s \in \cohen$ such that $s \force \theta (\dot x)$. Pick $T$ stationary-Silver tree of good Cohen branches over $\model{N}$ such that $\stem(T)=s$. Hence, for every $z \in [T]$, we have $\model{N}[z] \models \theta(z)$, and so 
\[
\model{N}[z] \models \text{``} \force_{\dot R^z} \varphi(z)\text{ ''}.
\] 
Since any $z$ has good quotient, it follows that $\dot R^z$ is $\cohen_{\kappa^+}$. That means that there exists $H$ filter $\dot R^z$-generic (i.e., $\cohen_{\kappa^+}$-generic) over $\model{N}[z]$ such that $\model{N}[z][H]=\model{N}[G]$. Hence $\model{N}[G] \models \varphi(z)$, that gives us $\model{N}[G] \models [T] \subseteq X$. 

For the case $\model{N}[x] \models \text{``}\not \force_{\dot R^x} \varphi(x)\text{''}$, simply note that $\text{``}\not \force_{\dot R^x} \varphi(x)\text{''}$ is equivalent to $\text{``}\force_{\dot R^x} \neg \varphi(x)\text{''}$, by weak homogeneity. Hence, a specular argument provides us with $T \in \statsilver$ such that $\model{N}[G] \models [T] \cap X = \emptyset$.
\end{proof}
\begin{remark}
Note that $\level(T)$ is both stationary and co-stationary. As a consequence, $[T]$ is completely disjoint both from $\club$ and from $\ns$, and so there is no contradiction with Lemma \ref{lemma:non-club-silver-meas}.
\end{remark}

\paragraph{A word about the Silver game.}
In the classical setting one can uniformly introduce an unfolding game associated with any notion of regularity coming from a certain tree forcing (see \cite{L98}). Here, we focus on the unfolding game connected to the Silver measurability. To this aim we need to introduce the ideal $\ideal{I}_\poset{V}$ of \emph{Silver small} sets.
\begin{definition}
$X \subseteq 2^\kappa$ is said to be $\silver$\emph{-null} iff for all $T \in \silver$ there exists $T' \leq T$, $T' \in \silver$ such that $[T'] \cap X=\emptyset$. Further, we define $\ideal{I}_\silver$ as the $\kappa^+$-ideal $\kappa^+$-generated by the $\silver$-null sets.
\end{definition}
For emulating the classical unfolding game, we need to satisfy, for every $X \subseteq 2^\kappa$,
\begin{itemize}
\item[(*)] if II has a winning strategy in $G(X)$ then $X \in \ideal{I}_\silver$; 
\item[(**)] if I has a winning strategy in $G(X)$ then there exists $T \in \silver$ such that $[T] \subseteq X$.
\end{itemize}

Nevertheless, in the context of $2^\kappa$ the situation seems to be less clear.  In our generalized setting, the output of the game has to be a $\kappa$-real, and so we consider games of length $\kappa$. The basic idea is the same as the standard case, i.e., player I and II play conditions such that each is stronger than the previous one. \emph{But} what should the rule be at limit steps? First of all, note that at limits it is natural to pick the intersection of all previous moves, and hence we want the forcing to be $<\kappa$-closed. This forces us to work with $\clubsilver$. We essentially have two choices, depending on \emph{who} chooses first at limit steps.

\begin{definition}
We use the following notation: \[ T' \preceq T \text{ iff } T' \leq T \text{ and } |\stem(T')| > |\stem(T)| .\]
Given $X \subseteq 2^\kappa$, we define two games $G_1(X)$ and $G_2(X)$ of length $\kappa$ as follows: for $n < \omega$, player I chooses $T^1_n \preceq T^2_{n-1}$, and player II chooses $T^2_n \preceq T^1_n$. From the first limit ordinal, $G_1(X)$ and $G_2(X)$ are defined differently: 
\begin{itemize}
\item[-] in $G_1(X)$ player I chooses first, i.e., player I first chooses $T^1_\omega \preceq \bigcap_{\xi < \omega} T^2_\xi$; then player II chooses $T^2_\omega \preceq T^1_\omega$. Then the game continues by following this order of choice (so in particular, at any limit $\lambda$, I chooses first).
\item[-] in $G_2(X)$ the situation is reversed: player II first chooses $T^2_\omega \preceq \bigcap_{\xi < \omega} T^2_\xi$; then player I chooses $T^1_\omega \preceq T^2_\omega$. Then the game continues by following this order of choice (so in particular, at any limit $\lambda$, II chooses first).
\end{itemize}
The output of the game will then be $x$ such that $\{ x \}:= \bigcap_\xi [T^1_\xi]$, and we will say that I wins iff $x \in X$, otherwise II wins.
\end{definition}
Unfortunately, both fail to have the desired properties (*) and (**) mentioned above. In fact, the reason for that is strictly connected to the bad behaviour of $\clubsilver$-measurability.
\begin{lemma} \label{lemma:G1-G2}
Player II has a winning strategy in $G_2(\club)$, while player I has a winning strategy in $G_1(\club)$.
\end{lemma}
\begin{proof}
We recursively construct the winning strategy of II in $G_2(\club)$ as follows: we only take care of limit steps $\lambda$: if $\langle T^1_\xi, T^2_\xi: \xi < \lambda \rangle$ is the partial play, then II chooses $T^2_\lambda \preceq \bigcap_{\xi < \lambda} T^2_\xi$ so that for $ \alpha_\lambda:= |\stem \big( \bigcap_{\xi < \lambda} T_\xi \big )|$, one has 
\begin{equation} \label{eq1}
|\stem(T^2_\lambda)| > \alpha_\lambda \land \stem(T^2_\lambda)(\alpha_\lambda)=0. 
\end{equation}
Note that one can make such a choice since $\stem (\bigcap_{\xi < \kappa} T_\xi)$ is a splitting node. Let us call $\sigma$ such a strategy for player II. For every $T^1_*:=\langle T^1_\xi: \xi < \kappa \rangle$ play of I, one has that the output produced by $\sigma (T^1_*)$ is not in $\club$, since the set of $\{ \alpha_\lambda : \lambda < \kappa \text{ limit ordinal} \}$ is closed unbounded.

To check the second assertion, we can analogously build the winning strategy $\tau$ for player I in $G_1(\club)$. Player I chooses first at limit steps $\lambda$, and so, in (\ref{eq1}), we can freely choose $\stem(T^1_\lambda)(\alpha_\lambda)=1$. In such a way, for every $T^2_* :=\langle T^2_\xi: \xi < \kappa \rangle$ play of II, one has the output produced by $\tau(T_*^2)$ is in $\club$.
\end{proof}

An interesting issue might be to switch the point of view in the following sense. Define $X \subseteq 2^\kappa$ to be $G_i$-measurable iff $G_i(X)$ is determined. 
By Lemma \ref{lemma:G1-G2}, the club filter $\club$ is measurable in both cases.

\begin{itemize}
\item[]\textsf{Question}. Can we force  all $\On^\kappa$-definable sets to be $G_i$-measurable? Or, in other words, can one find a model where $G_i$'s are determined for all $\On^\kappa$-definable sets?
\end{itemize}

\section{Full-Miller measurability} \label{miller}

In this section, we prove that $\levy(\kappa, < \lambda)$ forces that all $\On^\kappa$-definable subsets of $\kappa^\kappa$ are $\fullmiller$-measurable. We assume $2^\kappa=\kappa^+$.
Consider the forcing $\poset{MT}$ introduced in section \ref{perfect-cohen}, for adding a full-Miller tree of Cohen reals.

\begin{claim} \label{claim2}
$\poset{MT}$ is forcing-equivalent to $\levy(\kappa,2^\kappa)$.
\end{claim}

\begin{proof}
$\poset{MT}$ is clearly $< \kappa$-closed and has size $2^\kappa$. Moreover, $\poset{MT}$ collapses $2^\kappa$ to $\kappa$; in fact, for every  $A :=\{ a_\xi:\xi < \kappa \} \subseteq \kappa$ of size $\kappa$, $A \in \model{N}$, the set
\[
D_A :=  \{ \sigma \in \poset{MT}: \exists t \in \splitting(\sigma) \forall \xi < \kappa (t^\conc \xi^\conc a_\xi \in \sigma ) \}
\]
is open dense. Hence the function $H: \splitting(T_G) \rightarrow 2^\kappa \cap \model{N}$ defined by $H(t):= \{ \alpha: \exists \xi < \kappa (t^\conc \xi^\conc \alpha) \}$ is surjective, and so $2^\kappa \cap \model{N}$ collapses to $\kappa$. 

$\poset{MT}$ is then $< \kappa$-closed, of size $2^\kappa$, collapsing $2^\kappa$ to $\kappa$, and hence equivalent to $\levy(\kappa,2^\kappa)$.
\end{proof}

\begin{claim} \label{claim-miller1}
Let $\poset{Q}= \levy(\kappa, <\lambda)$, and let $\dot T$, $\dot x$ be $\poset{MT}*\poset{Q}$-names for the full-Miller generic tree added by $G(0)$ and a branch of $[\dot T]$, respectively. There exists $\poset{MT}_0 * \poset{P} \subseteq \poset{MT}*\poset{Q}$ dense subposet such that for every $(\sigma,\dot p) \in \poset{MT}_0*\poset{P}$ there exists $t \in \term(\sigma)$ such that $x_{(\sigma,\dot p)}=t$, where $x_{(\sigma,\dot p)}$ is the initial segment of $\dot x$ decided by $(\sigma,\dot p)$.
\end{claim}
\begin{proof}
First of all, we want to prove an analogue of Claim \ref{silver-claim1}. More precisely, we want to prove that  the set of conditions $(\sigma,\dot p)$ for which there exists $t \in \term(\sigma)$ such that $t \subseteq x_{(\sigma,\dot p)}$ is dense in $\poset{MT}* \poset{Q}$. To this aim, we start from a condition $(\sigma_0, \dot p_0)$ and we inductively build $(\sigma_{n+1},\dot p_{n+1}) \leq (\sigma_n, \dot p_n)$ such that there exists $t_{n} \in \term(\sigma_n)$ such that $x_{(\sigma_{n+1},\dot p_{n+1})} \supseteq t_n$. Then put $\sigma= \bigcup_{n \in \omega} \sigma_n$, pick $\dot p$ such that $\sigma \force \dot p \leq \dot p_n$ for all $n \in \omega$, and put $w= \bigcup_{n \in \omega} t_n$. By construction, $w \in \term(\sigma)$ and $x_{(\sigma, \dot p)} \supseteq w$, as $(\sigma,\dot p) \leq (\sigma_n,\dot p_n)$, for all $n \in \omega$. 

The second part is an analogue of the proof of Claim \ref{silver-claim2}, i.e., we want to show that if $x_{(\sigma,\dot p)} \supseteq t$, for some $t \in \term(\sigma)$, none of the extensions of $t$ can be ruled out, and so $t=x_{(\sigma,\dot p)}$.  By contradiction, assume $x_{(\sigma,p)}=t^\conc s$, for some $t \in \term(\sigma)$ and non-empty $s \in 2^{<\kappa}$. Let $\sigma'$ be the downward closure of $\sigma \cup \bigcup\{{t_0}^\conc \xi: \xi \in \kappa \}$, for some $t_0 \perp  t^\conc s$ with $t_0 \supset t$. Then pick $(\sigma'', \dot q) \leq (\sigma', \dot p)$ such that there exists $t_1 \in \term(\sigma'')$ such that $t_1 \subseteq x_{(\sigma'', \dot q)}$.
Hence, one has   $x_{(\sigma'',\dot q)} \supseteq t_1 \supseteq t_0 \perp x_{(\sigma, \dot p)}$, contradicting $(\sigma'', \dot q)\leq (\sigma, \dot p)$.
\end{proof}

With a similar construction, we can get an analogue of Claim \ref{silver-claim3} and Corollary \ref{cor:dense-cohen} as well.

\begin{claim} \label{quotient-miller}
Let $G$ be $\poset{MT}_0*\poset{P}$-generic over $\model{N}$. Let $\dot T$ be the canonical name for the generic Miller tree added by $G(0)$, $\dot x$ an $\poset{MT}_0*\poset{P}$-name for a branch in $\dot T$, and $z=\dot x^G$. Then $\poset{MT}_0*\poset{P} / \dot x=z$ is forcing-equivalent to $\poset{MT}_0*\poset{P}$ (and hence to $\levy(\kappa,<\lambda)$).
\end{claim}
\begin{proof}
Use the notation $(\poset{MT}_0*\poset{P})_z := \poset{MT}_0*\poset{P} / \dot x = z$. 
Claim \ref{claim-miller1}, together with the analogues of Claim \ref{silver-claim3} and Corollary \ref{cor:dense-cohen}, gives the same argument as in the proof of Lemma \ref{quotient-lemma}, and so we can obtain
\[
(\poset{MT}_0*\poset{P})_z = \{ (\sigma, \dot p) \in \poset{MT}_0*\poset{P}: x_{(\sigma,\dot p)} \subset z  \}.
\]
We work in $\model{N}[z]$. Note that 
\[
(\poset{MT}_0*\poset{P})_z =\{ (\sigma, \dot p) \in \poset{MT}_0*\poset{P}: \exists t \in \term(\sigma) (t \subset z \land x_{(\sigma,\dot p)}=t)\}.
\]

$\subseteq$: clearly, if $\forall t \in \term(\sigma) (t \not \subset z)$, then $(\sigma,\dot p) \notin (\poset{MT}_0*\poset{P})_z$, as $(\sigma,\dot p) \force \sigma \sqsubset \dot T$. 
$\supseteq$: if there exists $t \in \term(\sigma)$ such that $t \subset z$, then $x_{(\sigma, \dot p)}= t\subset z$.

First, we prove that 
\begin{equation} \label{eq2}
\poset{P}_0:= \{ \sigma \in \poset{MT}_0: \exists t \in \term(\sigma) \exists \dot p \in \poset{P} (t \subset z \land x_{(\sigma,\dot p)}=t)\}
\end{equation}
 is $<\kappa$-closed and collapses $2^\kappa$ to $\kappa$, and so it is equivalent to $\poset{MT}_0$.  
Let $\{ \sigma_\alpha: \alpha < \delta \}$, for $\delta < \kappa$, be a decreasing sequence of conditions in $\poset{P}_0$, and for every $\alpha < \delta$, let $t_\alpha \in \term(\sigma_\alpha)$ be such that $t_\alpha \subset z$ and $\dot p_\alpha \in \poset{P}$ such that $x_{(\sigma_\alpha, \dot p_\alpha)}=t_\alpha$. Then put $\sigma_\delta= \bigcup_{\alpha<\delta} \sigma_\alpha$, $t_\delta= \bigcup_{\alpha < \delta} t_\alpha$ and pick $\dot p_\delta \in \poset{P}$ such that $\sigma_\delta \force \forall \alpha < \delta (\dot p_\delta \leq \dot p_\alpha)$. Hence, $t_{\delta} \in \term(\sigma_\delta)$, $t_{\delta} \subset z$ and $t_\delta=x_{(\sigma_\delta, \dot p_\delta)}$, which means $\sigma_\delta \in \poset{P}_0$. Hence, the poset is $<\kappa$-closed. The proof that it also collapses $2^\kappa$ to $\kappa$ is the same as the one given for Claim \ref{claim2}, since the sets $D_A$'s are dense in $\poset{P}_0$ as well; simply, for every $\sigma \in \poset{P}_0$, pick $t \in \term(\sigma)$ such that $t \perp x_{(\sigma,\dot p)}$, for some $\dot p \in \poset{P}$, and then let $\sigma' \leq \sigma$ be the downward closure of $\sigma \cup \bigcup \{t^\conc \xi^\conc a_\xi: \xi \in \kappa  \}$, where $A:= \{ a_\xi:\xi \in \kappa \}$.

Secondly, define
\begin{equation} \label{eq3}
\dot{\poset{P}}_1:= \{\dot p \in \poset{P}: \exists \sigma \in \poset{MT}_0 ((\sigma, \dot p) \in (\poset{MT}_0*\poset{P})_z)  \}.
\end{equation}
 Let $H$ be an arbitrary $\poset{MT}_0$-generic filter over $\model{N}[z]$. Work in $\model{N}[z][H]$. Then $\levy(\kappa, <\lambda)$ is equivalent to $\poset{P}_1$. Indeed, first note that, the argument used in the second part of the proof of Claim \ref{claim-miller1} actually gives the following: if $\dot p, \dot q \in \poset{P}$ are such that $\sigma \force \dot q \leq \dot p$ and $(\sigma,\dot p) \in  (\poset{MT}_0*\poset{P})_z$, then $x_{(\sigma,\dot p)}= x_{(\sigma, \dot q)}$, and so $(\sigma, \dot q) \in (\poset{MT}_0*\poset{P})_z$ as well. (if we drop the assumption $\sigma \force \dot q \leq \dot p$, the only thing that we can say in general is that $\exists t_0 \in \term(\sigma)$ such that $x_{(\sigma, \dot p)}=t_0$ and $\exists t_1 \in \term(\sigma)$ such that $x_{(\sigma, \dot q)}=t_1$, but $t_0$ and $t_1$ might be different). Furthermore, let $\{ p_\xi: \xi < \delta \}$, for $\delta \leq \lambda$, be the set of minimal conditions in $\poset{P}_1$ (i.e., there is no $q \geq p_\xi$ and $q \neq p_\xi$ such that $q \in \poset{P}_1$); we can build a partial function $e : \levy(\kappa,<\lambda) \rightarrow \levy(\kappa, <\lambda)$, satisfying:
\begin{itemize}
\item for every $\xi < \delta$, for all $\alpha_0 \in \lambda$ and $\beta_0, \mu_0 \in \kappa$, there are $\alpha_0' \in \lambda$ and $\beta_0', \mu_0' \in \kappa$ such that, 
\[
e(p_\xi \cup \{((\alpha_0,\beta_0),\mu_0)  \})= \{ ((\alpha'_0,\beta_0'), \mu'_0) \};
\]
\item for all $\alpha_0' \in \lambda$ and $\beta_0', \mu_0' \in \kappa$, there are $\xi_0 < \delta$, $\alpha_0 \in \lambda$ and $\beta_0, \mu_0 \in \kappa$ such that
\[
e(p_\xi \cup \{((\alpha_0,\beta_0),\mu_0)  \})= \{ ((\alpha'_0,\beta_0'), \mu'_0) \};
\]
\item let $q_0:= p_{\xi_0} \cup \{((\alpha_0,\beta_0),\mu_0)$ and $q_1:= p_{\xi_1} \cup \{((\alpha_1,\beta_1),\mu_1)$. Then $q_1 \leq q_0 \Rightarrow e(q_1) \leq e(q_0)$ and $q_1 \perp q_0 \Rightarrow e(q_0) \perp e(q_1)$;
\item let $\poset{P}_2:= \poset{P}_1 \setminus \{p_\xi:\xi < \delta \}$; then $e | \poset{P}_2 : \poset{P}_2 \rightarrow \levy(\kappa, <\lambda)$ is a dense embedding.
\end{itemize} 
This $e$ can be constructed by a pretty standard argument, simply by following a bijection $\delta \times \lambda \times \kappa \times \kappa \leftrightarrow \lambda \times \kappa \times \kappa$, and by using the homogeneity of $\levy(\kappa, <\lambda)$.

Hence, (\ref{eq2}) and (\ref{eq3}) give: $(\poset{MT}_0*\poset{P})_z \cong \poset{P}_0 * \dot{\poset{P}}_1 \cong \levy(\kappa, <\lambda)$.
\end{proof}

\begin{lemma} \label{lemma1:miller}
Let $\lambda$ be inaccessible greater than $\kappa$, and let $G$ be $\levy(\kappa, < \lambda)$-generic over $\model{N}$. Then 
\[
\model{N}[G] \models \text{`` all $\On^\kappa$-definable subsets of $\kappa^\kappa$ are $\fullmiller$-measurable ''.}
\]
\end{lemma}
\begin{proof}
The argument is in strict analogy to the one of Lemma \ref{lemma:projective-silver-meas}, and we just give a sketch. Let $X \subseteq \kappa^\kappa$ be defined via some formula $\varphi$ whose parameters can be absorbed into the ground model $\model{N}$, by $\lambda$-cc. Let $T_0$ be the generic tree in $\fullmiller$ added by the first step, i.e., $T_0$ is associated with $G_0:= G \cap \levy(\kappa,\kappa^+)$. By Claim \ref{quotient-miller}, we know that each branch $x \in [T_0] \cap \model{N}[G]$ has good quotient, and so $\levy(\kappa,<\lambda)$ can be viewed as $\dot Q_x * \dot R$, where $\dot Q_x$ is the poset generated by $x$ and $\model{N}[x] \models \dot R^x \cong \levy(\kappa,<\lambda)$. 

Let $x$ be Cohen over $\model{N}$ with good quotient and assume $\model{N}[x] \models \text{``} \force_{\dot R^x} \varphi(x) \text{''}$. Work into $\model{N}[x]$; pick $s \in \kappa^{<\kappa}$ such that $s \force \text{``} \force_{\dot R^x} \varphi(x) \text{''}$ (here we are using $\cohen \cong (\kappa^{<\kappa}, \subset)$). Pick $T$ full-Miller tree of good Cohen branches with $\stem(T)=s$. Then proceed as in Lemma \ref{lemma:projective-silver-meas}: for every $z \in [T]$, $\model{N}[z] \models \text{``} \force_{\dot R^z} \varphi(z) \text{''}$, which implies there exists a $\levy(\kappa, < \lambda)$-generic filter $H$ over $\model{N}[z]$ with $\model{N}[z][H]=\model{N}[G]$, and so $\model{N}[G] \models \varphi(z)$, as $\dot R^z \cong \levy(\kappa,<\lambda)$.
The case $\model{N}[x] \models \text{``} \not \force_{\dot R^x} \varphi(x) \text{''}$ is analogous.
\end{proof}

\begin{remark}
Our result cannot be improved by replacing $T \in \fullmiller$ with trees having branches $z$ satisfying $\{\alpha<\kappa: y \restric \alpha \text{ is splitting}  \}$ being club. Indeed, a similar argument to the one presented in Lemma \ref{lemma:non-club-silver-meas} shows that $\club$ is not $\clubmiller$-measurable (see also \cite[Thm 2.12]{FKK14} for a more general approach). However it remains open whether one can get $\statmiller_\textsf{full}$-measurability for all $\On^\kappa$-definable subsets of $\kappa^\kappa$. 
\end{remark}

\begin{remark} \label{final-remark}
In \cite{LMS14}, the authors investigate two properties related to the Miller measurability: the Hurewicz dichotomy and a strengthening called the Miller tree Hurewicz dichotomy. These notions are related to the Miller measurability, but they are not in general equivalent. The authors of \cite{LMS14} prove that $\levy(\kappa,<\lambda)$ forces all $\On^\kappa$-definable sets to have the Hurewicz dichotomy. Furthermore, they prove that if $\kappa$ is not weakly compact, then the Miller tree Hurewicz dichotomy fails for closed sets, whereas that cannot be true for the Miller measurability because of Lemma \ref{lemma1:miller}. On the contrary, for $\kappa$ weakly compact, they prove that the two dichotomies are equivalent and they both imply the Miller measurability pointwise, but it is not clear which is the relation with the full-Miller measurability.
\end{remark}

\section{Open questions}
In section \ref{silver} we have proved that $\cohen_{\kappa^+}$ forces all $\On^\kappa$-definable sets to be stationary-Silver measurable, for $\kappa$ inaccessible. The latter assumption was essential in our proof to show that $\cohen$ adds a stationary Silver tree of Cohen branches. Therefore, the following question arises naturally.
\begin{itemize}
\item[Q.1] Can one force all $\On^\kappa$-definable sets to be stationary-Silver measurable, for $\kappa$ successor? 
\end{itemize}  
Even if not strictly necessary for a positive answer to Q.1, another issue strictly related is the following.
\begin{itemize}
\item[Q.2] Does $\cohen$ add a stationary Silver tree of Cohen branches even for $\kappa$ successor?
\end{itemize}
About Q.2, my intuition inclines to a negative answer.

Another interesting issue is the role of the inaccessible $\lambda$ concerning full-Miller measurability and Miller measurability.
\begin{itemize}
\item[Q.3] Can one force all $\On^\kappa$-definable sets to be Miller measurable without using inaccessible cardinals?  
\item[Q.4] What about the same question for full-Miller measurability instead?  
\end{itemize}
The key point here is that we do not have an analogous study in the classical setting; indeed, in the standard case, projective Baire property implies projective Miller measurability (and even projective full-Miller measurability) and so Shelah's amalgamation and sweetness provide us with a model for those notions of regularity without any need of an inaccessible. 
\emph{But}, in our generalized context, the Baire property fails for $\SSigma^1_1$, and hence we really need a direct method to get Miller measurability. A possible solution might be to consider an amoeba forcing adding a Miller tree of Cohen branches in a gentler way than $\levy(\kappa, 2^{\kappa})$, in order to get: 1) $\kappa^+$ will not be collapsed, and 2) one could obtain sufficiently many good Cohen branches.

The issue of separating different regularities classwise has been developed in the classical setting: in particular a method for separating Silver and Miller on all sets has been presented in \cite{L13}. A similar questions arises here.
\begin{itemize}
\item[Q.5] Can one force all sets to be Silver measurable but there exists a non-Miller measurable set?  
\end{itemize}
Finally, a last important research branch regards the $\DDelta^1_1$-level. In fact, because of the $\DDelta^1_1$-well ordering of $\model(\kappa^\kappa)^\model{L}$, one obtains $\DDelta^1_1$ non-regular sets in $\model{L}$. As a consequence, some arguments used in the standard setting for $\DDelta^1_2$ sets hold true for $\DDelta^1_1$ in the generalized context. The investigation of this topic has been initiated by Friedman, Khomskii and Kulikov in \cite{FKK14}. We also believe that this topic be strictly connected to the study of cardinal characteristics associated with the ideals generated by tree-forcings, and hence a careful study of the amoeba forcings is necessary. In the standard setting, amoeba forcings have been studied in \cite{Sp95} and \cite{L14}, where the authors have  presented  some applications to regularity properties and cardinal characteristics. In the generalized setting such a topic has not been suitably developed yet, and we aim at extending such an investigation.


\addcontentsline{toc}{section}{Bibliography}

\begin{thebibliography}{9}

\bibitem{Br99} J\"org Brendle,
                \emph{Mutual generics and perfect free subsets}, Acta Mathematica Hungarica (1999), Vol. 82, pp. 143-161. 
             
\bibitem{FZ10} Sy D. Friedman, Lyubomir Zdomskyy,
                \emph{Measurable cardinals and the cofinality of the symmetric group}, Fundamenta Matematicae, Vol. 207 (2010), pp 101-122.
                
\bibitem{FKK14} Sy D. Friedman, Yurii Khomskii, Vadim Kulikov,
                \emph{Regularity properties on the generalized reals}, preprint (2014).


\bibitem{HS01} Aapo Halko, Saharon Shelah,
                \emph{On strong measure zero subsets of $2^\kappa$},
                Fundamenta Matematicae 170(3) (2001), pp 219-229.
                
\bibitem{HN73} H. Hung, S. Negrepontis,
                \emph{Spaces homeomorphic to $(2^\alpha)_\alpha$},
                Bull. Amer. Math. Soc. 79 (1973), pp 143-146.               

\bibitem{L13} Giorgio Laguzzi,
                \emph{On the separation of regularity properties of the reals}, submitted. 

\bibitem{L14} Giorgio Laguzzi,
                \emph{Some considerations on amoeba forcing notions}, Archive for Mathematical Logic (2014), DOI: 10.1007/s00153-014-0375-x.
                
\bibitem{L98} Benedikt L\"owe,
                \emph{Uniform unfolding and analytic measurability},
                Archive for Mathematical Logic 97 (1998), pp 505-520. 
                
\bibitem{LMS14} Philipp L\"ucke, Luca Motto Ros, Philipp Schlicht,
                \emph{The Hurewicz dichotomy for generalized Baire space},
                preprint (2014). 
                
\bibitem{Sc13}Philipp Schlicht,
                \emph{Perfect subsets in the generalized Baire space}, preprint (2013).

\bibitem{Sp95} Otmar Spinas,
                \emph{Generic trees}, Journal of Symbolic Logic (1995), Vol. 60, No. 3, pp. 705-726. 
                               

                
               
                
                 
                             
\end{thebibliography}
\end{document}